\documentclass[11pt]{amsart}
\usepackage{amsfonts}
\usepackage{}
\usepackage{pifont}
\usepackage{mathtools}
\usepackage{esint}
\usepackage{cite}
\usepackage{amsmath}
\usepackage{amssymb}
\usepackage{latexsym}
\usepackage{amscd}
\usepackage{mathrsfs}
\usepackage[all]{xy}

\usepackage{hyperref}
\usepackage{color}

\newdimen\AAdi%
\newbox\AAbo%
\def\AAk#1#2{\s_etbox\AAbo=\hbox{#2}\AAdi=\wd\AAbo\kern#1\AAdi{}}%
\def\AAr#1#2#3{\s_etbox\AAbo=\hbox{#2}\AAdi=\ht\AAbo\raise#1\AAdi\hbox{#3}}%
\font\tenmsb=msbm10 at 12pt \font\sevenmsb=msbm7 at 8pt
\font\fivemsb=msbm5 at 6pt
\newfam\msbfam
\textfont\msbfam=\tenmsb \scriptfont\msbfam=\sevenmsb
\scriptscriptfont\msbfam=\fivemsb

\newtheorem{theorem}{Theorem}

\newtheorem{remark}[theorem]{Remark}

\newtheorem{lemma}[theorem]{Lemma}

\numberwithin{equation}{section} \numberwithin{theorem}{section}

\renewcommand{\topmargin}{0cm}
\renewcommand{\oddsidemargin}{5mm}
\renewcommand{\evensidemargin}{5mm}
\renewcommand{\textwidth}{150mm}
\renewcommand{\textheight}{230mm}

\def\R{\mathbb R}

\def\N{\mathbb N}

\def\na{\nabla}

\def\f#1#2{\frac{#1}{#2}}

\def\a{\alpha}
\def\be{\beta}

\def\r{\Re_{I\!V}}

\def\p#1{\partial #1}

\def\de{\delta}
\def\De{\Delta}
\def\e{\eta}
\def\ep{\epsilon}

\def\g{\gamma}

\def\la{\lambda}

\def\lan{\langle}
\def\ran{\rangle}

\def\th{\theta}
\def\Th{\Theta}
\def\si{\sigma}
\def\Si{\Sigma}

\def\r{\rho}
\def\z{\zeta}

\begin{document}

\title[Liouville theorem for minimal graphs]
{Liouville theorem for minimal graphs over manifolds 
of nonnegative Ricci curvature}

\author{Qi Ding}
\address{Shanghai Center for Mathematical Sciences, Fudan University, Shanghai 200438, China}
\email{dingqi@fudan.edu.cn}

\thanks{The author is supported by NSFC 12371053}

\begin{abstract}
Let $\Si$ be a complete Riemannian manifold of nonnegative Ricci curvature. 
We prove a Liouville-type theorem: every smooth solution $u$ to minimal hypersurface equation on $\Si$ is a constant provided $u$ has sublinear growth for its negative part. Here, the sublinear growth condition is sharp.
Our proof relies on a gradient estimate for minimal graphs over $\Si$ with small linear growth of the negative parts of graphic functions via iteration.

\end{abstract}

\maketitle

\section{Introduction}

Let $\Si$ be a complete non-compact Riemannian manifold.
Let $D,\mathrm{div}_\Si$ be the Levi-Civita connection and the divergence operator (in terms of the Riemannian metric of $\Si$), respectively.
In this paper, we study the minimal hypersurface equation on $\Si$
\begin{equation}\label{u0}
\mathrm{div}_\Si\left(\f{Du}{\sqrt{1+|Du|^2}}\right)=0,
\end{equation}
which is a non-linear partial differential equation describing the minimal graph
$$M=\{(x,u(x))\in \Si\times\R|\, x\in\Si\}$$
over $\Si$. 
The smooth solution $u$ to \eqref{u0} is the height function of the minimal graph $M$ in $\Si\times\R$. Therefore, we call $u$ a \emph{minimal graphic function} on $\Si$.

When $\Si$ is Euclidean space $\R^n$, the equation \eqref{u0} has been studied successfully by many mathematicians.
In 1969, Bombieri-De Giorgi-Miranda \cite{BGM} (see also \cite{GT}) proved interior gradient estimates for solutions to the minimal hypersurface equation on $\R^n$, where the 2-dimensional case had already been obtained by Finn \cite{F}.
As a corollary, they immediately got a Liouville type theorem in \cite{BGM} as follows.
\begin{theorem}\label{SLE}
If a minimal graphic function $u$ on $\R^n$ satisfies sublinear growth for its negative part, i.e.,
\begin{equation}\label{sublinear00}
\limsup_{x\to\infty}\f{\max\{-u(x), 0\}}{|x|}=0,
\end{equation}
then $u$ is a constant.
\end{theorem}
The condition \eqref{sublinear00} is sharp since any affine function is a minimal graphic function on $\R^n$.
When the minimal graphic function $u$ on $\R^n$ has the uniformly bounded gradient, Moser can prove $u$ affine using Harnack's inequalities for uniformly elliptic equations \cite{M}. The gradient estimate of $u$ on $\R^n$ can also be derived by the maximum principle (see \cite{K,W} for instance).
Without the 'uniformly bounded gradient' condition, it is the celebrated Bernstein theorem (see \cite{Fl,DG,Al,Si} and the counter-example in \cite{BDG}). Specially, any  minimal graphic function on $\R^n$ is affine for $n\le7$ by Simons \cite{Si}.

Let us review Liouville type theorems for nonnegative minimal graphic functions on manifolds briefly. 
From \cite{FS} by Fischer-Colbrie and Schoen, any positive minimal graphic function on a Riemann surface $S$ of nonnegative curvature is constant (see Rosenberg \cite{R} for the case of minimal surfaces in $S\times\R$).
In 2013, Rosenberg-Schulze-Spruck proved that every nonnegative minimal graphic function on a complete manifold of nonnegative Ricci curvature and sectional curvature uniformly bounded below, is a constant.
Casteras-Heinonen-Holopainen \cite{CHH} showed that every nonnegative
minimal graphic function $u$ on a complete manifold of asympototically nonnegative sectional curvature is a constant provided $u$ has at most linear growth. In \cite{D0}, the author proved that every nonnegative minimal graphic function on a complete manifold of nonnegative Ricci curvature is constant, which was also obtained by Colombo-Magliaro-Mari-Rigoli \cite{CMMR} independently.
In fact, the 'nonnegative Ricci curvature' condition can be further weakened to volume doubling property and Neumann-Poincar\'e inequality in \cite{D0}.

In some situations, the above 'nonnegative' condition for the solution $u$ on a manifold $\Si$ can be weaken to the condition of 'sub-linear growth for its
negative part', i.e.,
\begin{equation}\label{sublinear}
\limsup_{\Si\ni x\to\infty}\f{\max\{-u(x), 0\}}{d(x,p)}=0
\end{equation}
for some $p\in\Si$, where $d(x,p)$ denotes the distance function on $\Si$ between $x,p$.
Motivated by Theorem \ref{SLE}, for shorthand we say the \emph{strong Liouville theorem for minimal graphs over} $\Si$ if
every minimal graphic function $u$ on $\Si$ is a constant provided $u$ admits sub-linear growth for its negative part.

In \cite{RSS}, Rosenberg-Schulze-Spruck proved the strong Liouville theorem for minimal graphs over complete manifolds of nonnegative sectional curvature. Ding-Jost-Xin \cite{DJX1} proved the strong Liouville theorem for minimal graphs over complete manifolds of nonnegative Ricci curvature, Euclidean volume growth and quadratic curvature decay. 
In \cite{D3}, the author further proved it without the above quadratic curvature decay condition,
which is a bi-product of Poincar\'e inequality on minimal graphs (see \cite{BG} for the Euclidean case). In \cite{CGMR}, Colombo-Gama-Mari-Rigoli proved the strong Liouville theorem for minimal graphs over complete manifolds of nonnegative Ricci curvature and that the $(n-2)$-th Ricci curvature in radial direction from a fixed origin has a lower bound decaying quadratically to zero.

Recently, Colombo-Mari-Rigoli \cite{CMR} proved an interesting theorem that if a minimal graphic function $u$ on a complete non-compact Riemannian manifold $\Si$ of nonnegative Ricci curvature, satisfies 
\begin{equation}\label{dlogd}
\limsup_{\Si\ni x\to\infty}\f{\log d(x,p)}{d(x,p)}\max\{-u(x), 0\}<\infty
\end{equation}
for some $p\in\Si$, then $u$ is a constant.

From now on, we always let $\Si$ denote a complete non-compact Riemannian manifold of nonnegative Ricci curvature (without extra assumptions).
In this paper, we can weaken the condition \eqref{dlogd} to \eqref{sublinear}, and prove the strong Liouville theorem for minimal graphs over $\Si$ as follows.
\begin{theorem}\label{PMGFC0}
Any minimal graphic function $u$ on $\Si$ is a constant provided $u$ has sublinear growth for its negative part.
\end{theorem}
The condition of 'sublinear growth for its negative part', i.e., \eqref{sublinear}, is sharp from the Euclidean case and the manifolds case (see Proposition 9 in \cite{CGMR}).
For getting Theorem \ref{PMGFC0}, we prove a more stronger result: a gradient estimate for small linear growth of the negative part of $u$ (without the upper bound condition of $u$) as follows.
\begin{theorem}\label{SmallG}
There exists a constant $\be_*>0$ depending only on $n$ such that if a minimal graphic function $u$ on $\Si$ satisfies
\begin{equation}\label{-be*0}
\liminf_{x\to\infty}\f{u(x)}{d(x,p)}\ge-\be_*
\end{equation}
for some $p\in\Si$,
then there is a constant $c>0$ depending only on $n$ such that
\begin{equation}\label{GGE}
\sup_{x\in\Si}|Du|(x)\le c\limsup_{x\to\infty}\f{\max\{-u(x), 0\}}{d(x,p)}.
\end{equation}
\end{theorem}
The key ingredient in the proof of Theorem \ref{SmallG} is to get an integral estimate of $v^k$ on geodesic balls in $\Si$ for a large constant $k$ by an iteration (on $l$) of an integral of $(\log v)^lv$, where $v$ is the volume function of the minimal graphic function $u$. Then using Sobolev inequality on $\Si$, we can carry out a (modified) De Giorgi-Nash-Moser iteration on geodesic balls in $\Si$ starting from an integral of $v^{2k}$ with $k>n$, and get the bound of $v$ (see Theorem \ref{SmaG} since the Harnack's inequality holds in Theorem 4.3 of \cite{D0}).

Once we get the uniform gradient estimate \eqref{GGE}, from Theorem 8 (or Theorem 6	(ii)) in \cite{CGMR} we can conclude that any tangent cone of $\Si$ at infinity splits off a line isometrically, compared with the harmonic case by Cheeger-Colding-Minicozzi \cite{CCM}.
It's worth to point out that $\Si$ may not split off any line from a counterexample in Proposition 9 of \cite{CGMR}.

Without \eqref{-be*0}, we had the gradient estimates without the 'entire' condition of $M$ or $\Si$, where the estimates depend on the lower bound of the volume of geodesic balls of $\Si$ (see \cite{D3}). In \cite{CGMR}, Colombo-Gama-Mari-Rigoli obtained gradient estimates for minimal graphs over manifolds of nonnegative Ricci curvature and that
the $(n-2)$-th Ricci curvature of $\Si$ in radial direction from a fixed origin has a lower bound decaying quadratically to zero.

\section{Preliminary}

Let $\Si$ be an $n$-dimensional complete Riemannian manifold of nonnegative Ricci curvature.
For any $R>0$ and $p\in\Si$, let $B_{R}(p)$ be the geodesic ball in $\Si$ centered at $p$ with radius $R$.
For each integer $k\ge0$,  
let $\mathcal{H}^k$ denote the $k$-dimensional Hausdorff measure.
From Bishop-Gromov volume comparison theorem,
\begin{equation}\label{VolD}
\f1nr^{1-n}\mathcal{H}^{n-1}(\p B_r(p))\le r^{-n}\mathcal{H}^{n}(B_r(p))\le s^{-n}\mathcal{H}^{n}(B_{s}(p))
\end{equation}
for all $0<s<r$. Let $D$ be the Levi-Civita connection of $\Si$.
From Anderson \cite{An1} or Croke \cite{Cr}, the Sobolev inequality 
\begin{equation}\label{Sob1}
\f{\left(\mathcal{H}^{n}(B_r(p))\right)^{\f1{n}}}r\left(\int_{B_r(p)}|\phi|^{\f n{n-1}}\right)^{\f{n-1}n}\le\Th\int_{B_r(p)}|D\phi|
\end{equation}
holds for any Lipschitz function $\phi$ on $B_r(p)$ with compact support in $B_r(p)$, where $\Th>0$ is a constant depending only $n$.

Let $\Phi$ be a Lipschitz function on $B_{r+s}(p)$, $s\in(0,r]$, and $\z$ be a nonnegative Lipschtz function so that $\z\equiv1$ on $B_r(p)$, $\z\equiv0$ outside $B_{r+s}(p)$ and $|D\z|\le\f1s$. Then from Cauchy-Schwarz inequality, we have
\begin{equation}\aligned
&\int_{B_r(p)}|D(\Phi^2\z)|\le2\int_{B_r(p)}|\Phi|\z|D\Phi|+\int_{B_r(p)}\Phi^2|D\z|\\
\le&r\int_{B_r(p)}|D\Phi|^2\z+\f1r\int_{B_r(p)}\Phi^2\z+\f{1}{s}\int_{B_{r+s}(p)}\Phi^2.
\endaligned\end{equation}
From \eqref{Sob1}, it follows that 
\begin{equation}\label{Sob2}
\left(\mathcal{H}^{n}(B_r(p))\right)^{\f1{n}}\left(\int_{B_r(p)}|\Phi|^{\f {2n}{n-1}}\right)^{\f{n-1}n}\le\Th r^2\int_{B_{r}(p)}|D\Phi|^2+ \f{2\Th r}s\int_{B_{r+s}(p)}\Phi^2.
\end{equation}
From Buser \cite{Bu} or Cheeger-Colding \cite{CC}, there holds the Neumann-Poincar\'e inequality on geodesic balls of $\Si$.
Namely, it holds (up to a choice of $\Th$)
\begin{equation}\label{NPoincare}
\int_{B_r(p)}|\varphi-\varphi_{B_r(p)}|\le \Th r\int_{B_r(p)}|D \varphi|
\end{equation}
for any Lipschitz function $\varphi$ on $B_r(p)$, where
$$\varphi_{B_r(p)}=\fint_{B_r(p)}\varphi:=\f1{\mathcal{H}^n(B_r(p))}\int_{B_r(p)}\varphi.$$

Let $M$ be a minimal graph over $\Si$ with the graphic function $u$ on $\Si$, where $M$ has the induced metric from $\Si\times\R$ equipped with the standard product metric. By Stokes' formula, $M$ is area-minimizing in $\Si\times\R$ analog to an argument as in Euclidean space.
Let $\na$ and $\De$ denote the Levi-Civita connection and Laplacian of $M$, respectively.
We also see $u$ as a function on $M$ by projection $M\to\Si$, i.e., $u(x,u(x))=u(x)$ for any $x\in\Si$.
Then the equation \eqref{u0} is equivalent to that $u$ is harmonic on $M$, i.e.,
\begin{equation}\label{DeMu0}
\De u=0.
\end{equation}
Let 
$$v=\sqrt{1+|Du|^2}$$ be the volume function of $M$ (as mentioned above), and we see $v$ as a function on $M$ by identifying $v(x,u(x))=v(x)$.
Recalling the following Bochner type formula:
\begin{equation}\aligned\label{Dev-1}
\De v^{-1}=-\left(|A|^2+v^{-2}Ric(Du,Du)\right)v^{-1}.
\endaligned
\end{equation}
Here, $A$ denotes the second fundamental form of $M$ in $\Si\times\R$, and $Ric$ denotes the Ricci curvature of $\Si$.
From \eqref{Dev-1}, it follows that
\begin{equation}\aligned\label{Delog v}
\De\log v=|A|^2+v^{-2}Ric(Du,Du)+|\na\log v|^2\ge|\na\log v|^2.
\endaligned
\end{equation}
For a $C^1$-function $f$ on an open set of $\Si$, we can see $f$ as a function on $M$ by letting $f(x,u(x))=f(x)$. Then
\begin{equation}\aligned\label{nafDf}
|\na f|^2=|Df|^2-\f{1}{v^2}|\lan Du,Df\ran|^2\ge|Df|^2-\f{|Du|^2}{v^2}|Df|^2=\f{1}{v^2}|Df|^2.
\endaligned
\end{equation}

$\mathbf{Notional\ convention}.$ 
When we write an integation on a subset of a Riemannian manifold w.r.t. its standard metric of the manifold,
we always omit the volume element for simplicity.

\section{Integral estimates of powers of the volume function}

\begin{lemma}\label{DlogvmxiEST}
Let $\xi$ be a Lipschitz function on $\Si$ with compact support. For all constants $l\ge1$ and $q,\th>0$ we have
\begin{equation}\aligned\label{logm}
\int_\Si|D(\log v)^l|\xi^{q+1}\le l\th r\int_{\Si}(\log v)^{l-1}v|D\xi|^2+\f l{\th r}\int_{\Si}(\log v)^{l-1}v\xi^{2q}.
\endaligned
\end{equation}
\end{lemma}
\begin{proof}
We also see $\xi$ as a function on $M$ by letting $\xi(x,u(x))=\xi(x)$. From \eqref{Delog v},
for each $l'\ge0$ from Cauchy-Schwarz inequality we have
\begin{equation}\aligned\label{logmxi}
\int_M(\log v)^{l'}\xi^2|\na\log v|^2
\le&\int_M(\log v)^{l'}\xi^2\De\log v
\le-2\int_M(\log v)^{l'}\xi\na\xi\cdot\na\log v\\
\le&\f12\int_M(\log v)^{l'}\xi^2|\na\log v|^2+2\int_M(\log v)^{l'}|\na\xi|^2,
\endaligned
\end{equation}
which implies
\begin{equation}\aligned\label{logvm'}
\int_M(\log v)^{l'}|\na\log v|^2\xi^2\le4\int_M(\log v)^{l'}|\na\xi|^2.
\endaligned
\end{equation}
From \eqref{nafDf} and \eqref{logvm'}, for all constants $q,\th>0$ and $l\ge1$ we have
\begin{equation}\aligned\label{Dlogmxi}
\int_\Si|D(\log v)^l|\xi^{q+1}\le&\int_M|\na(\log v)^l|\xi^{q+1}=l\int_M(\log v)^{l-1}|\na\log v|\xi^{q+1}\\
\le&\f{l\th r}4\int_M(\log v)^{l-1}|\na\log v|^2\xi^2+\f l{\th r}\int_{M}(\log v)^{l-1}\xi^{2q}\\
\le&l\th r\int_{M}(\log v)^{l-1}|\na\xi|^2+\f l{\th r}\int_{M}(\log v)^{l-1}\xi^{2q}.
\endaligned
\end{equation}
This gives \eqref{logm} by combining with \eqref{nafDf} again.
\end{proof}

Given two constants $\be,r_0>0$, we assume
\begin{equation}\aligned\label{Growu0}
|u(x)|\le\be \max\{r_0,d(x,p)\}\qquad \mathrm{for\ each}\ x\in\Si.
\endaligned
\end{equation} 
For each $r\ge r_0$, it's clear that 
\begin{equation}\aligned\label{Growu}
|u(x)|\le\be \max\{r,d(x,p)\}\qquad \mathrm{for\ each}\ x\in\Si.
\endaligned
\end{equation} 
We fix the point $p$ and denote $\r(x)=d(x,p)$ for each $x\in\Si$.
\begin{lemma}
Given a constant $\th\in(0,1]$ and a constant $0<\de<<1$, for each constant $l\ge1$ we have
\begin{equation}\aligned\label{loglvsmalll}
\fint_{B_{r}(p)}(\log v)^lv\le&(1+\de)\be l\f{(1+\th)^{n+1}}{\th}\fint_{B_{(1+\th)r}(p)}(\log v)^{l-1}v\\
&+\f{2^n(1+c_\de\be)}{\th}\fint_{B_{(1+\th)r}(p)} (\log v)^l,
\endaligned
\end{equation}
where $c_\de\ge1$ is a constant depending only on $n,\de$.
\end{lemma}
\begin{proof}
Let $\de$ be a positive constant with $\de<<1$, and $\xi$ be a Lipschitz function on $\Si$ with supp$\xi\subset B_{(1+\th)r}(p)$, $\xi\equiv1$ on $B_{r}(p)$ and 
\begin{equation}
\xi(x)=\left\{\begin{split}
\cos\left(\f{\r(x)-r}{\th r}\right)\qquad\qquad &\mathrm{for}\ x\in B_{(1+\de\th/4)r}(p)\setminus B_r(p)\\
\f{\cos(\de/4)}{1-\de/4}\left(1-\f {\r(x)-r}{\th r}\right)\qquad &\mathrm{for}\ x\in B_{(1+\th)r}(p)\setminus B_{(1+\de\th/4)r}(p)\\
\end{split}\right..
\end{equation}
Then
\begin{equation}\label{xidef}
\th r|D\xi|(x)=\left\{\begin{split}
\sin\left(\f{\r(x)-r}{\th r}\right)\qquad &\mathrm{for}\ x\in B_{(1+\de\th/4)r}(p)\setminus B_r(p)\\
\f{\cos(\de/4)}{1-\de/4}\quad\qquad &\mathrm{for}\ x\in B_{(1+\th)r}(p)\setminus B_{(1+\de\th/4)r}(p)\\
\end{split}\right..
\end{equation}
Let $q=q_\de>1$ so that 
$$\cos^q(\de/4)=\f{\sin(\de/4)}{1-\de/4}.$$
Noting $(1-\de/4)^{-2}<1+\de$ as $0<\de<<1$, with \eqref{xidef} we have
\begin{equation}\aligned\label{xiprop}
\th^2r^2|D\xi|^2+\xi^{2q}\le(1-\de/4)^{-2}<1+\de\qquad \mathrm{on}\ \Si.
\endaligned
\end{equation}

Bombieri-De Giorgi-Miranda \cite{BGM} gave an estimate of an integral of $v\log v$ using \eqref{u0} in the Euclidean case (see also \cite{GT}, and \cite{DJX1} for manifolds). Enlightened by this,
we further estimate an integral of $(\log v)^lv$ on geodesic balls of $\Si$ using \eqref{u0} for each $l>0$.
Integrating by parts, for each $r\ge r_0$ with \eqref{Growu} we have
\begin{equation}\aligned\label{Duvlogve}
0=&\int_\Si\f{Du}{v}\cdot D(u(\log v)^l\xi^{q+1})\\
=&\int_\Si\f{|Du|^2}v(\log v)^l\xi^{q+1}+\int_\Si u\xi^{q+1}\f{Du}{v}\cdot D(\log v)^l+\int_\Si u(\log v)^l\f{Du}{v}\cdot D\xi^{q+1}\\
\ge&\int_\Si\f{|Du|^2}v(\log v)^l\xi^{q+1}-(1+\th)\be r\int_\Si\xi^{q+1}|D(\log v)^l|-\f{c_\de\be}{\th}\int_{B_{(1+\th)r}(p)} (\log v)^l.
\endaligned
\end{equation}
Here, $c_\de\ge1$ is a constant depending only on $n,q=q_\de$.
Then it follows that
\begin{equation}\aligned\label{logvlve0}
\int_\Si(\log v)^lv\xi^{q+1}\le&\int_\Si\left(\f{|Du|^2}v+1\right)(\log v)^l\xi^{q+1}\\
\le&(1+\th)\be r\int_\Si|D(\log v)^l|\xi^{q+1}+\f{1+c_\de\be}{\th}\int_{B_{(1+\th)r}(p)} (\log v)^l.
\endaligned
\end{equation}
For each $l\ge1$, from \eqref{logm} and \eqref{xiprop}, we have
\begin{equation}\aligned\label{Dlogvle}
\int_\Si|D(\log v)^l|\xi^{q+1}\le& \f l{\th r}\int_{B_{(1+\th)r}(p)}(\log v)^{l-1}v\left(\th^2r^2|D\xi|^2+\xi^{2q}\right)\\
\le&\f{(1+\de)l}{\th r}\int_{B_{(1+\th)r}(p)}(\log v)^{l-1}v.
\endaligned
\end{equation}
Substituting \eqref{Dlogvle} into \eqref{logvlve0}, we get
\begin{equation}\aligned
\int_{B_{r}(p)}(\log v)^lv\le(1+\de)\be l\f{(1+\th)}{\th}\int_{B_{(1+\th)r}(p)}(\log v)^{l-1}v+\f{1+c_\de\be}{\th}\int_{B_{(1+\th)r}(p)} (\log v)^l.
\endaligned
\end{equation}
This finishes the proof with \eqref{VolD}.
\end{proof}

Now we further assume $\be\le1$. Denote $\g_\de=(1+\de)n(1+\f1n)^{n+1}$. 
By taking $\th=1/n$ in \eqref{loglvsmalll}, for each $l\ge1$ (up to a choice of $c_\de\ge1$) we have
\begin{equation}\aligned\label{logvlva}
\fint_{B_{r}(p)}(\log v)^lv\le \g_\de\be l\fint_{B_{\f{(n+1)r}n}(p)}(\log v)^{l-1}v+c_\de\fint_{B_{\f{(n+1)r}n}(p)} (\log v)^l.
\endaligned
\end{equation}
Since $M$ is area-minimizing in $\Si\times\R$, with \eqref{VolD} we get
\begin{equation}\aligned\label{volv}
\int_{B_r(p)}v=&\mathcal{H}^n(M\cap(B_r(p)\times\R))\le \mathcal{H}^n(B_r(p))+\int_{\p B_r(p)}|u|\\
\le&\mathcal{H}^n(B_r(p))+\be r\mathcal{H}^n(\p B_r(p))\le (1+n\be)\mathcal{H}^n(B_r(p)).
\endaligned
\end{equation}

Let us iterate the estimate \eqref{logvlva} on $l$.
\begin{lemma}\label{Logvlv}
Let $c_\de$ be the constant in \eqref{logvlva} with the given $0<\de<<1$.
For each integer $j\ge0$ there holds
\begin{equation}\aligned\label{logvlv}
\sup_{r\ge r_0}\fint_{B_{ r}(p)}(\log v)^jv\le j!\g_\de^j\be^j\tbinom{j+m}{m}(1+n\be),
\endaligned
\end{equation}
where $m=[\f{c_\de}{\g_\de\be}]+1\in\N$ depends on $n,\de,\be$, and $\tbinom{j+m}{m}=\f{(m+j)!}{j!m!}$.
\end{lemma}
\begin{proof}
Let us prove it by induction. 
From \eqref{logvlva} and $\log v\le v$, for each $j\ge1$ we have
\begin{equation}\aligned\label{suprr0logvlv}
\sup_{r\ge r_0}\fint_{B_{r}(p)}(\log v)^jv\le \g_\de\be j\sup_{r\ge r_0}\fint_{B_r(p)}(\log v)^{j-1}v+c_\de\sup_{r\ge r_0}\fint_{B_{r}(p)}(\log v)^{j-1}v.
\endaligned
\end{equation}
Let $m=[\f{c_\de}{\g_\de\be}]+1\in\N$ depend on $n,\de,\be$, and $\{a_j\}_{j\in\N}$ be a sequence defined by
\begin{equation}\aligned
a_j=\sup_{r\ge r_0}\fint_{B_{r}(p)}(\log v)^{j}v.
\endaligned
\end{equation}
From \eqref{suprr0logvlv}, for each integer $j\ge1$ one has
\begin{equation}\aligned
a_j\le \g_\de\be ja_{j-1}+c_\de a_{j-1}\le \g_\de\be(j+m)a_{j-1}.
\endaligned
\end{equation}
By iteration,
\begin{equation}\aligned\label{allogvlv}
a_j\le \g_\de^j\be^j\f{(j+m)!}{m!}a_0=j!\g_\de^j\be^j\tbinom{j+m}{m}a_0.
\endaligned
\end{equation}
From \eqref{volv}, $a_0\le1+n\be$.
This completes the proof.
\end{proof}

\begin{theorem}\label{fintvla+1}
Let $u$ be a minimal graphic function on $\Si$ satisfying \eqref{Growu} for some constant $\be\in(0,1]$. There is a constant $c(n,\de,\be)>0$ depending only on $n,\de,\be$ so that for each constant $\la\in(0,\f1{\g_\de\be})$ we have
\begin{equation}\aligned\label{vla}
\sup_{r\ge r_0}\fint_{B_{ r}(p)}v^{\la+1}\le c(n,\de,\be)(1-\la \g_\de\be)^{-m-1}.
\endaligned
\end{equation}
\end{theorem}
\begin{proof}
Let $\la$ be a positive constant $<\f1{\g_\de\be}$.
From Taylor's expansion
\begin{equation}\aligned
v^\la=e^{\la\log v}=\sum_{j=0}^\infty\f{\la^j}{j!}(\log v)^j,
\endaligned
\end{equation}
combining with \eqref{logvlv}
we get
\begin{equation}\aligned\label{vla}
\fint_{B_{ r}(p)}v^{\la+1}=&\sum_{j=0}^\infty\f{\la^j}{j!}\fint_{B_{ r}(p)}(\log v)^j v\le \sum_{j=0}^\infty\f{\la^j}{j!}j!\g_\de^j\be^j\tbinom{j+m}{m}(1+n\be)\\
=& (1+n\be)\sum_{j=0}^\infty(\la \g_\de\be)^j\tbinom{j+m}{m}.
\endaligned
\end{equation}
From
\begin{equation}\aligned
\sum_{j=0}^\infty \tbinom{j+m}{m}t^j=\f1{m!}\f{d^{m}}{d t^{m}}\sum_{j=0}^\infty t^{j+m}=\f1{m!}\f{d^{m}}{d t^{m}}\left(\f{t^{m}}{1-t}\right)
\endaligned
\end{equation}
for each $t\in(0,1)$, we complete the proof.
\end{proof}

\section{Mean value inequality and gradient estimate}

For each nonnegative measurable function $f$ on $\Si$ and each constant $q>0$, we denote
$$||f||_{q,r}=\left(\fint_{B_r(p)}f^q\right)^{1/q}.$$
Now, let us carry out a (modified) De Giorgi-Nash-Moser iteration to get the mean value inequality for the volume function $v$ with the help of
the Sobolev inequality on $\Si$.
\begin{lemma}\label{vinftyder}
For each constant $k>n$ and $\si\in(0,1)$, there is a constant $c_{\si,k}$ depending only on $n,\si,k$ such that
\begin{equation}\aligned\label{subharmkge2}
||v||_{\infty,\si r}\le c_{\si,k}\left(||v||_{2k,r}\right)^{e^{\f{n}{k-n}}}
\endaligned
\end{equation}
for any $r>0$.
\end{lemma}
\begin{proof}
Let $\e$ be a Lipschitz function on $\Si$ with compact support to be defined later. 
Denote $\e(x)=\e(x,u(x))$.
From \eqref{Delog v}, there holds
$\De v\ge0$ on $M$ clearly. For any constant $\ell\ge1$, we have
\begin{equation}\aligned\label{v2le2Dev}
0\ge&-\int_M v^{2\ell}\e^2\De v=2\ell\int_M v^{2\ell-1}\e^2|\na v|^2+2\int_M v^{2\ell}\e\na v\cdot\na\e\\
\ge&2\ell\int_M v^{2\ell-1}\e^2|\na v|^2-\ell\int_M v^{2\ell-1}\e^2|\na v|^2-\f1{\ell}\int_M v^{2\ell+1}|\na\e|^2\\
=&\ell\int_M v^{2\ell-1}\e^2|\na v|^2-\f1{\ell}\int_M v^{2\ell+1}|\na\e|^2.
\endaligned
\end{equation}
From \eqref{nafDf} and \eqref{v2le2Dev}, it infers that
\begin{equation}\aligned\label{naellle2}
\int_\Si\left|Dv^\ell\right|^2\e^2\le&\int_M\left|\na v^\ell\right|^2v\e^2=\ell^2\int_M\left|\na v\right|^2v^{2\ell-1}\e^2\\
\le&\int_M v^{2\ell+1}|\na\e|^2\le\int_\Si v^{2\ell+2}|D\e|^2.
\endaligned
\end{equation}
For each $r\ge\tau>0$, let $\e$ be defined by $\e\equiv1$ on $B_{r+\f\tau2}(p)$, $\e=\f2\tau\left(r+\tau-\r\right)$ on $B_{r+\tau}(p)\setminus B_{r+\f\tau2}(p)$,
$\e\equiv0$ outside $B_{r+\tau}(p)$.
Then $|D\e|\le2/\tau$.
Combining \eqref{Sob2} and \eqref{naellle2}, we have
\begin{equation}\aligned\label{v2lnr}
||v^{2\ell}||_{\f {n}{n-1},r}\le& \Th\left(r^2||D v^{\ell}||_{2,r+\f\tau2}^2+\f{4r}{\tau}||v^{2\ell}||_{1,r+\f\tau2}\right)\\
\le& \Th\left(r^2\int_\Si v^{2\ell+2}|D\e|^2+\f{4r}{\tau}||v^{2\ell}||_{1,r+\tau}\right)\\
\le& \Th\left(\f{4r^2}{\tau^2}||v^{2\ell+2}||_{1,r+\tau}+\f{4r}{\tau}||v^{2\ell}||_{1,r+\tau}\right)\\\le& c \f{r^2}{\tau^2}||v^{2\ell+2}||_{1,r+\tau}=c \f{r^2}{\tau^2}||v||_{2\ell+2,r+\tau}^{2\ell+2}.
\endaligned
\end{equation}
Here, $c=8\Th$ is a constant depending only on $n$.
Given a constant $k>n$, we set \
\begin{equation}\aligned\label{Defa}
\a=\f{n(k-1)}{(n-1)k}>1.
\endaligned
\end{equation}
For $\ell+1\ge k$, we have
\begin{equation}\aligned\label{ite}
\f{2\ell n}{n-1}-(2\ell+2)\a=\f{2n}{(n-1)k}(\ell+1-k)\ge0.
\endaligned
\end{equation}
From H\"older inequality and \eqref{v2lnr}, one has
\begin{equation}\aligned\label{ite}
||v||_{(2\ell+2)\a,r}\le ||v||_{\f{2\ell n}{n-1},r}\le c^{\f1{2\ell}}r^{\f1{\ell}}\tau^{-\f1{\ell}}||v||_{2\ell+2,r+\tau}^{\f{\ell+1}{\ell}}.
\endaligned
\end{equation}

For any $\si\in(0,1)$ and any integer $i\ge-1$, set $m_i=2k\a^i$, $\ell_i=m_i/2-1$, $\tau_i=2^{-(1+i)}(1-\si)r$ and
$r_{i+1}=r_{i}-\tau_{i+1}$ with $r_{-1}=r$.
Then
$$r_{i+1}=r-\sum_{j=0}^{i+1}\tau_j=\si r+\tau_{i+1}\le r,$$
and $\lim_{i\to\infty}r_i=\si r$.
By iterating \eqref{ite}, for each $i\ge0$ we have
\begin{equation}\aligned
||v||_{\a m_{i},r_i}\le c^{\f1{2\ell_i}}r_i^{\f1{\ell_i}}\tau_i^{-\f1{\ell_i}}||v||_{\a m_{i-1},r_{i-1}}^{\f{\ell_i+1}{\ell_i}}.
\endaligned
\end{equation}
Set $\xi_i=\log||v||_{\a m_{i},r_i}$ for each integer $i\ge-1$, and $b_\si=\f{c}{(1-\si)^2}$. Note that $\tau_i/r_i\ge2^{-(1+i)}(1-\si)$, and $\ell_i\ge k\a^i-1\ge (k-1)\a^i$ for every $i\ge0$. Then
\begin{equation}\aligned
\xi_i\le& \f1{2\ell_i}\log c+\f1{\ell_i}\log\f{r_i}{\tau_i}+\f{\ell_i+1}{\ell_i}\xi_{i-1}\le\f1{2\ell_i}\log b_\si+\f{1+i}{\ell_i}\log2+e^{\f1{\ell_i}}\xi_{i-1}\\
\le&\f1{2(k-1)\a^i}\log b_\si+\f{1+i}{(k-1)\a^i}\log2+e^{\f{\a^{-i}}{k-1}}\xi_{i-1}.
\endaligned
\end{equation}
For all $0\le i_0\le i$, it holds
$$\prod_{j=i_0}^i e^{\f{\a^{-j}}{k-1}}=e^{\f1{k-1}\sum_{j=i_0}^i\a^{-j}}\le e^{\f{\a^{1-i_0}}{(k-1)(\a-1)}}.$$
Hence, for each $i\ge1$
\begin{equation}\aligned
\xi_i\le&\f{\log b_\si}{2(k-1)\a^i}+\f{(1+i)\log2}{(k-1)\a^i}+e^{\f{\a^{-i}}{k-1}}\left(\f{\log b_\si}{2(k-1)\a^{i-1}}+\f{i\log2}{(k-1)\a^{i-1}}+e^{\f{\a^{1-i}}{k-1}}\xi_{i-2}\right)\\
\le&\cdots\le\sum_{j=0}^i\left(\f{\log b_\si}{2(k-1)\a^j}+\f{1+j}{(k-1)\a^j}\log2\right)\prod_{\jmath=j+1}^i e^{\f{\a^{-\jmath}}{k-1}}+\xi_{-1}\prod_{j=0}^i e^{\f{\a^{-j}}{k-1}}\\
\le&e^{\f{1}{(k-1)(\a-1)}}\sum_{j=0}^i\left(\f{\log b_\si}{2(k-1)}\f1{\a^j}+\f{\log2}{k-1}\f{1+j}{\a^j}\right)+e^{\f{\a}{(k-1)(\a-1)}}\xi_{-1}.
\endaligned
\end{equation}
Since 
\begin{equation}\aligned
\sum_{j=0}^\infty \f{j+1}{\a^j}=\f{\a^2}{(\a-1)^2},
\endaligned
\end{equation}
we have
\begin{equation}\aligned
\xi_i\le e^{\f{1}{(k-1)(\a-1)}}\left(\f{\log b_\si}{2(k-1)}\f{\a}{\a-1}+\f{\log2}{k-1}\f{\a^2}{(\a-1)^2}\right)+e^{\f{\a}{(k-1)(\a-1)}}\xi_{-1}.
\endaligned
\end{equation}
From $\a-1=\f{k-n}{(n-1)k}$ and $\f\a{\a-1}=\f{n(k-1)}{k-n}$,
we obtain
\begin{equation}\aligned
\xi_i\le e^{\f{n}{k-n}}\left(\f{n\log b_\si}{2(k-n)}+\f{n^2k\log 2}{(k-n)^2}\right)+e^{\f{n}{k-n}}\xi_{-1}.
\endaligned
\end{equation}
Namely,
\begin{equation}\aligned
||v||_{\a m_{i},r_i}\le\exp\left( e^{\f{n}{k-n}}\left(\f{n\log b_\si}{2(k-n)}+\f{n^2k\log 2}{(k-n)^2}\right)\right)\left(||v||_{2k,r}\right)^{e^{\f{n}{k-n}}}.
\endaligned
\end{equation}
Letting $i\rightarrow\infty$, it follows that
\begin{equation}\aligned\label{Phiinftyder}
||v||_{\infty,\si r}\le\exp\left( e^{\f{n}{k-n}}\left(\f{n\log b_\si}{2(k-n)}+\f{n^2k\log 2}{(k-n)^2}\right)\right)\left(||v||_{2k,r}\right)^{e^{\f{n}{k-n}}}.
\endaligned
\end{equation}
This completes the proof.
\end{proof}
\begin{remark}
The factor $e^{\f{n}{k-n}}$ in \eqref{subharmkge2} comes from \eqref{naellle2}, which transforms an estimate on $M$ to another estimate on $\Si$ with a slight but definite 'loss'.
In fact, the factor could be smaller if we choose a larger factor than $\a$ in \eqref{ite} for large $\ell$. However, we can not reduce the constant $k$ to a constant $\le n$ since we need $\a>1$ in \eqref{Defa}. Hence, unlike the classic De Giorgi-Nash-Moser iteration, here we are not able to obtain $\sup_{B_r(p)}v$ bounded by a multiple of an integral of $v^\g$ with $\g\le 2n$ on $B_{2r}(p)$.
\end{remark}

Put
\begin{equation}\label{be*}
\be_n=\f1{n(2n-1)}\left(1+\f1n\right)^{-n-1}.
\end{equation}
For proving Theorem \ref{SmallG}, we only need to show the following theorem since we have the Harnack's inequality in Theorem 4.3 of \cite{D0} (or \eqref{supBRphatuR*} in Appendix directly).
\begin{theorem}\label{SmaG}
If a minimal graphic function $u$ on $\Si$ satisfies
\begin{equation}\label{Guben1}
\limsup_{x\to\infty}\f{|u(x)|}{d(x,p)}<\be_n
\end{equation}
for some $p\in\Si$,
then there is a constant $c>0$ depending only on $n$ such that
\begin{equation}\label{Dubetter}
\sup_{x\in\Si}|Du|(x)\le c\limsup_{x\to\infty}\f{|u(x)|}{d(x,p)}.
\end{equation}
\end{theorem}
\begin{proof}
From \eqref{Guben1}, there is a constant $\be\in(0,\be_n)$ so that 
\begin{equation}
\limsup_{x\to\infty}\f{|u(x)|}{d(x,p)}<\be.
\end{equation}
Then there is a constant $r_\be>0$ so that 
\begin{equation}\aligned
|u(x)|\le\be \max\{r_\be,d(x,p)\}\qquad \mathrm{for\ each}\ x\in\Si.
\endaligned
\end{equation} 
We fix a positive constant $\de=\de(\be)<<1$ satisfying $\be(1+\de)<\be_n$.
Recalling $\g_\de=(1+\de)n(1+\f1n)^{n+1}$.
From Theorem \ref{fintvla+1}, there is a constant $\la_\be=\left(1+\f{\be_n}{\be(1+\de)}\right)(n-1/2)+1$ so that 
\begin{equation}\aligned
\fint_{B_{ r}(p)}v^{\la_\be}\le\f{c(n,\de,\be)}{(1-(\la_\be-1)\g_\de\be)^{m+1}}=c(n,\de,\be)\left(\f{2\be_n}{\be_n-(1+\de)\be}\right)^{m+1}
\endaligned
\end{equation}
for all $r\ge r_\be$.
From Lemma \ref{vinftyder}, we get
\begin{equation}\aligned\label{supBr2v}
\sup_{B_{r/2}(p)}v=||v||_{\infty,r/2}\le c_{\f12,\f{\la_\be}2}\left(||v||_{\la_\be,r}\right)^{e^{\f{n}{\la_\be/2-n}}}\le\psi(n,\be),
\endaligned
\end{equation}
where $\psi=\psi(n,\be)$ is a positive function depending only on $n$ and $\be<\be_n$ satisfying $\lim_{\be\to\be_n}\psi(n,\be)=\infty$, which may change from line to line.
In other words, we have concluded that $v$ is uniformly bounded on $\Si$. 
In the following, let us give a better bound of $v$ than \eqref{supBr2v}.

Let $\bar{p}=(p,u(p))$, and $B_{r}(\bar{p})$ denote the geodesic ball in $\Si\times\R$ with radius $r$ and centered at $\bar{p}$.
From (3.5) in \cite{D2} and \eqref{VolD}\eqref{volv}, we get
\begin{equation}\aligned\label{LowVolM}
2\mathcal{H}^{n}(B_{r}(p))\ge\mathcal{H}^n(M\cap B_{r}(\bar{p}))\ge \f1r\mathcal{H}^{n+1}(B_{r/2}(\bar{p}))\ge\f1{c}\mathcal{H}^{n}(B_{r}(p))
\endaligned
\end{equation}
for each $r>0$. Here, $c\ge1$ is a constant depending only on $n$, which may change from line to line.
Combining \eqref{Sob1} and \eqref{supBr2v}, (by projection from $\Si\times\R$ into $\Si$) there holds the Sobolev inequality on $M$, i.e.,
\begin{equation}
\left(\fint_{M\cap B_r(\bar{p})}|\phi|^{\f n{n-1}}\right)^{\f{n-1}n}\le \psi r\fint_{M\cap B_r(\bar{p})}|D\phi|
\end{equation}
holds for any Lipschitz function $\phi$ on $M\cap B_r(\bar{p})$ with compact support in $M\cap B_r(\bar{p})$. 
Combining \eqref{NPoincare} and \eqref{supBr2v}, there holds the Neumann-Poincar\'e inequality on exterior geodesic balls of $M$, i.e.,
\begin{equation}
\int_{M\cap B_r(\bar{p})}|\varphi-\bar{\varphi}_{p,r}|\le \psi r\int_{M\cap B_r(\bar{p})}|D \varphi|
\end{equation}
for any Lipschitz function $\varphi$ on $M\cap B_r(\bar{p})$ with 
$\bar{\varphi}_{p,r}=\fint_{M\cap B_r(\bar{p})}\varphi.$
From De Giorgi-Nash-Moser iteration, the mean value inequalities hold on $M$ for sub(super)-harmonic functions on $M$.
Denote $|Du|_0=\sup_\Si|D u|$. Since $|Du|^2$ is subharmonic on $M$ from \eqref{Dev-1}, we conclude that $|Du|_0^2-|Du|^2$ is nonnegative superharmonic on $M$, then (see page 42 in \cite{D3}, or Lemma 3.5 in \cite{DJX1} up to a suitable modification)
\begin{equation}\aligned\label{|Du|0equ}
|Du|_0^2=\sup_\Si|Du|^2=\lim_{r\to\infty}\fint_{M\cap B_{r}(\bar{p})}|D u|^2.
\endaligned
\end{equation}

Let $\tilde{\e}$ be a Lipschitz function on $\Si$ with supp$\tilde{\e}\subset B_{2r}(p)$, $\tilde{\e}\equiv1$ on $B_{r}(p)$ and $|D\tilde{\e}|\le\f1r$.
We also see $\tilde{\e}$ as a function on $M$ by letting $\tilde{\e}(x,u(x))=\tilde{\e}(x)$.
From \eqref{DeMu0} and Cauchy-Schwarz inequality, it follows that
\begin{equation}\aligned
0=&\int_M \na u\cdot\na(u\tilde{\e}^2)=\int_M|\na u|^2\tilde{\e}^2+2\int_Mu\tilde{\e}\na u\cdot\na\tilde{\e}\\
\ge&\int_M|\na u|^2\tilde{\e}^2-\f12\int_M|\na u|^2\tilde{\e}^2-2\int_Mu^2|\na\tilde{\e}|^2.
\endaligned
\end{equation}
Combining this with \eqref{VolD}\eqref{volv}, we get
\begin{equation}\aligned\label{B2rpv-1}
&\int_{B_{r}(p)}|\na u|^2v\le\int_M|\na u|^2\tilde{\e}^2\le 4\int_Mu^2|\na\tilde{\e}|^2\le16\be^2\int_{B_{2r}(p)}v\\
\le& 16(1+n\be)\be^2\mathcal{H}^n(B_{2r}(p))\le16(1+n\be)2^n\be^2\mathcal{H}^n(B_{r}(p)).
\endaligned
\end{equation}
Since $M\cap B_r(\bar{p})\subset B_r(p)\times\R$,
combining with \eqref{nafDf}\eqref{LowVolM}\eqref{|Du|0equ}\eqref{B2rpv-1} we get
\begin{equation}\aligned
&\f{|Du|_0^2}{1+|Du|_0^2}\le\limsup_{r\to\infty}\fint_{M\cap B_{r}(\bar{p})}\f{|D u|^2}{v^2}\le\limsup_{r\to\infty}\fint_{M\cap B_{r}(\bar{p})}|\na u|^2\\
\le& \limsup_{r\to\infty}\f1{\mathcal{H}^n(M\cap B_{r}(\bar{p}))}\int_{B_{r}(p)}|\na u|^2v\le c\limsup_{r\to\infty}\fint_{B_{r}(p)}|\na u|^2v\le c\be^2.
\endaligned
\end{equation}
Letting $\be\to\limsup_{x\to\infty}d^{-1}(x,p)|u(x)|$, we deduce \eqref{Dubetter}, which completes the proof.
\end{proof}

\section{Appendix}

Let $\Si$ be an $n$-dimensional complete Riemannian manifold of nonnegative Ricci curvature.
Let $M$ be a minimal graph over $\Si$ with the graphic function $u$ on $\Si$. Suppose $u$ is not a constant.
For any $r>0$ and $\bar{x}=(x,t_x)\in\Si\times\R$, we define
$$\mathfrak{D}_{\bar{x},r}=\{(y,s)\in\Si\times\R|\, d(y,x)+|s-t_x|<r\},$$
and
$\mathscr{B}_{r}(\bar{x})=M\cap \mathfrak{D}_{\bar{x},r}$.
For each $s\le\inf_{B_{4R}(p)}u$, denote $\bar{p}_s=(p,u(p)-s)$. Since $u-s>0$ on $B_{4R}(p)$ from the maximum principle for \eqref{u0}. From Theorem 4.3 in \cite{D0}, $u-s$ satisfies Harnack's inequality as follows:
\begin{equation}\aligned\label{Harnack}
\sup_{\mathscr{B}_{2R}(\bar{p}_s)}(u-s)\le\vartheta\inf_{\mathscr{B}_{2R}(\bar{p}_s)}(u-s)
\endaligned
\end{equation}
for some constant $\vartheta\ge2$ depending only on $n$.

We suppose that there is a positive constant $\be_*<\f{\be_n}{4(\vartheta-1)}$ with $\be_n$ defined as in \eqref{be*} so that 
\begin{equation}\label{GubenApp}
\liminf_{x\to\infty}\f{u(x)}{d(x,p)}\ge-\be_*
\end{equation}
for some $p\in\Si$. Denote $\ep=\f{\be_n}{8\be_*(\vartheta-1)}-\f12>0$. There is a constant $r_\ep>0$ so that
\begin{equation}\aligned
u(x)\ge-(1+\ep)\be_* \max\{d(x,p),r_\ep\}
\endaligned
\end{equation}
for all $x\in\Si$.
For each $R\ge r_\ep$, let 
$\hat{u}_R=u+4(1+\ep)\be_*R$ and $\hat{p}_R=(p,\hat{u}_R(p))\in\Si\times\R$, then $\hat{u}_R>0$ on $B_{4R}(p)$, which follows that
\begin{equation}\aligned\label{HarhatuR}
\sup_{\mathscr{B}_{2R}(\hat{p}_R)}\hat{u}_R\le\vartheta\inf_{\mathscr{B}_{2R}(\hat{p}_R)}\hat{u}_R\le\vartheta\hat{u}_R(p)
\endaligned
\end{equation}
from \eqref{Harnack}. Since
\begin{equation}\aligned
(\vartheta-1)\hat{u}_R(p)=(\vartheta-1)(u(p)+4(1+\ep)\be_*R)=(\vartheta-1) u(p)+(\be_n+4(\vartheta-1)\be_*)\f R2,
\endaligned
\end{equation}
we get
\begin{equation}\aligned
(\vartheta-1)\hat{u}_R(p)<\be_nR
\endaligned
\end{equation}
for the sufficiently large $R\ge r_\ep$.
Noting $B_R(p)\times(-R+\hat{u}_R(p),R+\hat{u}_R(p))\subset\mathfrak{D}_{\hat{p}_R,2R}$. From \eqref{HarhatuR}, we conclude that
\begin{equation}\aligned\label{supBRphatuR}
\sup_{B_R(p)}\hat{u}_R<\be_nR
\endaligned
\end{equation}
for all the sufficiently large $R\ge r_\ep$. From \eqref{supBRphatuR} and the definition of $\hat{u}_R$, it follows that
\begin{equation}\aligned\label{supBRphatuR*}
\sup_{B_R(p)}u<\sup_{B_R(p)}\hat{u}_R-4\be_*R<(\be_n-4\be_*)R.
\endaligned
\end{equation}

\bibliographystyle{amsplain}

\begin{thebibliography}{10}

\bibitem{Al} F. J. Almgren, Jr., Some interior regularity theorems for minimal surfaces and an extension of Bernstein's theorem, Ann. of Math. {\bf 85} (1966), 277-292.

\bibitem{An1} M. T. Anderson, The $L^2$ structure of moduli spaces of Einstein metrics on 4-manifolds, GAFA, {\bf 2} (1992), 29-89.

\bibitem{BGM} E. Bombieri, E. De Giorgi and M. Miranda. Una maggiorazione a priori relativa alle ipersuperfici minimali non parametriche, Arch. Rational Mech. Anal. {\bf 32} (1969), 255-267.

\bibitem{BG} E. Bombieri, E. Giusti, Harnack's inequality for elliptic differential equations on minimal surfaces, Invent. Math. {\bf 15} (1972), 24-46.

\bibitem{BDG} E. Bombieri, E. De Giorgi, and E. Giusti, Minimal cones and the Bernstein problem, Invent. Math. {\bf 7}(1969), 243-268.

\bibitem{Bu} P. Buser, A note on the isoperimetric constant, Ann. Scient. Ec. Norm. Sup. {\bf 15}(1982), 213-230.

\bibitem{CHH} Jean-Baptiste Casteras, Esko Heinonen, Ilkka Holopainen, Existence and non-existence of minimal
graphic and p-harmonic functions, Proc. Roy. Soc. Edinburgh Sect. A {\bf150} (2020), no. 1, 341-366.

\bibitem{CC} Jeff Cheeger and Tobias H. Colding, Lower Bounds on Ricci Curvature and the Almost Rigidity of Warped Products, Ann. Math. {\bf 144} (1996), 189-237.

\bibitem{CCM} Jeff Cheeger, Tobias H. Colding, William P. Minicozzi II, Linear growth harmonic functions on complete manifolds with nonnegative Ricci curvature, Geom. Funct. Anal. {\bf 5} (1995), no. 6, 948-954.

\bibitem{CGMR} G. Colombo, E. S. Gama, L. Mari and M. Rigoli, Non-negative Ricci curvature and minimal graphs of linear growth, to appear in Analysis $\&$ PDE, arXiv:2112.09886.


\bibitem{CMR} Giulio Colombo, Luciano Mari, Marco Rigoli, On minimal graphs of sublinear growth over manifolds with non-negative Ricci curvature, arXiv:2310.15620.

\bibitem{CMMR} Giulio Colombo, Marco Magliaro, Luciano Mari, Marco Rigoli, Bernstein and half-space properties for minimal graphs under Ricci lower bounds, Int. Math. Res. Not. IMRN(2022), no. 23, 18256-18290.

\bibitem{Cr} C. Croke, Some isoperimetric inequalities and eigenvalue estimates, Ann. Scient. Ec. Norm. Sup. 4, T {\bf 13} (1980), 419-435.

\bibitem{DG} E. De Giorgi, Una estensione del teorema di Bernstein, Ann. Sc. Norm. Sup. Pisa {\bf 19} (1965), 79-85.

\bibitem{D0} Qi Ding, Liouville-type theorems for minimal graphs over manifolds, Analysis $\&$ PDE {\bf 14}(6) (2021), 1925-1949.

\bibitem{D2} Qi Ding, Area-minimizing hypersurfaces in manifolds of Ricci curvature bounded below, J. Reine. Angew. Math. {\bf 798} (2023), 193-236.

\bibitem{D3} Qi Ding, Poincar\'e inequality on minimal graphs over manifolds and applications, arXiv:2111.04458.

\bibitem{DJX1} Qi Ding, J.Jost and Y.L.Xin, Minimal graphic functions on manifolds of non-negative Ricci curvature, Comm. Pure Appl. Math. {\bf 69}(2) (2016), 323-371.

\bibitem{F} R. Finn, On equations of minimal surface type, Ann. of Math., {\bf 60(2)} (1954), 397-416.

\bibitem{Fl} W. Fleming, On the oriented Plateau problem, Rend Circ. Mat. Palermo {\bf 11} (1962), 1-22.

\bibitem{FS} D. Fischer-Colbrie, R. Schoen, The structure of complete stable minimal surfaces in 3-manifolds of nonnegative scalar curvature, Commun. Pure Appl. Math. {\bf 33} (1980), 199-211.

\bibitem{GT} D. Gilbarg and N. Trudinger, Elliptic Partial Differential Equations of Second Order, Springer-Verlag, Berlin-New York, (1983).

\bibitem{K} N. Korevaar, An easy proof of the interior gradient bound for solutions of the precribed mean curvature equation, Proc. of Symp. in Pure Math., {\bf45}, Part 2, AMS (1986).

\bibitem{M} J$\mathrm{\ddot{u}}$rgen Moser, On Harnack's Theorem for Elliptic Differential Equations, Comm. Pure Appl. Math. {\bf 14} (1961), 577-591.

\bibitem{R} H. Rosenberg, Minimal surfaces in $M^2\times\R$, Illinois J. Math. {\bf 46} (2002), 1177-1195.

\bibitem{RSS} Harold Rosenberg, Felix Schulze and Joel Spruck, The half-space property and entire positive minimal graphs in $M\times\R$, J. Diff. Geom. {\bf 95} (2013), 321-336.

\bibitem{Si} J. Simons, Minimal varieties in Riemannian manifolds, Ann. Math. {\bf 88} (1968), 62-105.

\bibitem{W} Xu-Jia Wang, Interior gradient estimates for mean curvature equations, Mathematische Zeitschrift,
{\bf228} (1998), 73-81.


\end{thebibliography}

\end{document}